\documentclass[a4paper,10pt]{article}
\usepackage{amsfonts,amsmath}
\usepackage[utf8]{inputenc}
\usepackage[T1]{fontenc}
\usepackage{graphicx}
\usepackage{geometry}
\usepackage{float}
\usepackage{amsmath}
\usepackage{amsthm}
\usepackage{amsfonts}
\usepackage{amssymb}
\usepackage{mathdots}
\usepackage{stmaryrd}
\usepackage{xcolor}
\newtheorem{prop}{Proposition}[section]
\newtheorem{lem}{Lemma}[section]
\newtheorem{thm}{Theorem}[section]
\newtheorem*{theo}{Theorem}
\newtheorem{rem}{Remark}[section]

\newtheorem*{conje}{Conjecture}
\newcommand{\ts}{\textsuperscript}

\title{Distinction of the Steinberg representation and a conjecture of Prasad and Takloo-Bighash}
\author{Marion Chommaux}
\begin{document}
\maketitle

\begin{abstract}
We prove a conjecture of Prasad and Takloo-Bighash for discrete series of inner forms of the general linear group over a non archimedean local field, in the case 
of Steinberg representations.
\end{abstract}

 \section*{Introduction}
 
 Let $F$ be a non archimedean local field of characteristic not $2$ and $E$ a quadratic extension of $F$. Let $D$ be a central division $F$-algebra 
 of dimension $d^2$. If $n$ is a positive integer, then $\mathcal{M}_n(D)$ is a simple central $F$-algebra of dimension $n^2d^2$. As $E/F$ is quadratic, $E$ is embedded in $\mathcal{M}_n(D)$ as an $F$-subalgebra 
 if and only if $nd$ is even. 
 Then $C_{\mathcal{M}_n(D)}(E)$ (the centralizer of $E$ in 
 $\mathcal{M}_n(D)$) is an $E$-algebra. We denote by $N_{rd,F}$ (resp. $N_{rd,E}$) the reduced norm of $GL(n,D)$ (resp. $(C_{\mathcal{M}_n(D)}(E))^\times$) as well as its restriction to any subgroup.
 \paragraph{}We consider the following conjecture of Prasad and Takloo-Bighash: 
\begin{conje}[\cite{P-TB}, Conjecture 1]
Let $A=\mathcal{M}_n(D)$. Let $\pi$ be an irreducible admissible representation of $A^\times$ such that the corresponding representation (via Jacquet-Langlands) of $GL(nd,F)$ is generic with central 
character $\omega_\pi$. Let $\mu$ be a character of $E^\times$ such that $\mu^{\frac{nd}{2}}|_{F^\times}=\omega_\pi$. If the character $\mu\circ N_{rd,E}$ of $(C_{\mathcal{M}_n(D)}(E))^\times$ appears as a quotient in $\pi$ restricted to $(C_{\mathcal{M}_n(D)}(E))^\times$, then :
\begin{enumerate}
\item the Langlands parameter of $\pi$ takes values in $GSp_{nd}(\mathbb{C})$ with similitude factor $\mu_{|F^\times}$.
\item the epsilon factor $\epsilon(\frac{1}{2},\pi\otimes Ind_E^F(\mu^{-1}))=(-1)^n\omega_{E/F}(-1)^{\frac{nd}{2}}\mu(-1)^{\frac{nd}{2}}$ ($\omega_{E/F}$ is the quadratic character of $F^\times$ with kernel the norms of $E^\times$).
\end{enumerate}
If $\pi$ is a discrete series representation of $A^\times$, then these two conditions are necessary and sufficient for the character $\mu\circ N_{rd,E}$ of $(C_{\mathcal{M}_n(D)}(E))^\times$ to appear as a quotient in $\pi$ restricted to $(C_{\mathcal{M}_n(D)}(E))^\times$.
\end{conje}

Let us set $G=GL(n,D)$ and $H=(C_{\mathcal{M}_n(D)}(E))^\times$. We consider this conjecture for the Steinberg representation and we prove it by first  
establishing some $H$-distinction results. The main results of this paper are :

\begin{theo}
Let $n$ be a positive integer and let $\mu$ be a character of $E^*$. $E$ is embedded in $\mathcal{M}_n(D)$ if and only if $nd$ is even. We set $St(1)=St(n,1)$ the Steinberg representation of $G$ and 
$\tilde{\mu}:=\mu\circ N_{rd,E}$.
\begin{itemize}
\item If $d$ is even, $H=(C_{\mathcal{M}_n(D)}(E))^\times=GL(n,C_D(E))$ and $St(n,1)$ is $\tilde{\mu}$-distinguished under $H$ if and only if
\begin{itemize}
\item $\mu_{|F^*}=1$ and $\mu\neq1$ if $n$ is even.
\item $\mu=1$ if $n$ is odd.
\end{itemize}
\item If $d$ is odd and $n$ is even, $H=(C_{\mathcal{M}_n(D)}(E))^\times=GL(n/2,D\otimes_FE)$ and $St(n,1)$ is $\tilde{\mu}$-distinguished under $H$ if and only if $\mu_{|F^*}=1$ and $\mu\neq1$.
\end{itemize}
\end{theo}

\begin{theo}(Prasad and Takloo-Bighash conjecture, Steinberg case) Let $A=\mathcal{M}_n(D)$ and $\pi=St(n,1)$ which is an irreducible admissible representation of $A^\times=GL(n,D)=G$. Recall that $\pi$ corresponds via Jacquet-Langlands correspondence to 
$St(nd,1)$ (the Steinberg representation of $GL(nd,F)$) with central character $\omega_\pi=1$. Let $\mu$ be a character of $E^\times$ such that $\mu^{\frac{nd}{2}}|_{F^\times}=\omega_\pi=1$. Then, the character $\mu\circ N_{rd,E}$ of 
$H=(C_{\mathcal{M}_n(D)}(E))^\times$ appears as a quotient in $\pi$ restricted to $H$ if and only if :
\begin{enumerate}
 \item the Langlands parameter of $\pi$ takes values in $GSp_{nd}(\mathbb{C})$ with similitude factor $\mu_{|F^\times}$.
 \item the epsilon factor satisfies $\epsilon(\frac{1}{2},\pi\otimes Ind_E^F(\mu^{-1}))=(-1)^n\omega_{E/F}(-1)^{\frac{nd}{2}}$ (where $\omega_{E/F}$ is the quadratic character of $F^\times$ with kernel the norms of $E^\times$).
\end{enumerate}
\end{theo}

Notice that when the character $\mu\circ N_{rd,E}$ of $H$ extends to $G$, the conjecture can be reformulated in the following more appealing way, and this is the version of the conjecture considered in \cite{FMW}.

\begin{theo}
\begin{sloppypar}
  Let $St(n,\chi)$ be the Steinberg representation of $G=GL(n,D)$. $St(n,\chi)$ is $H$-distinguished if and only if it is symplectic and $\epsilon(\frac{1}{2},BC_E(St(n,\chi)))=(-1)^n$ (where $BC_E$ denotes the base change to $E$).
\end{sloppypar}
\end{theo}

In this latter reference, the authors partially prove the conjecture 
for supercuspidal representations of $GL(n,\mathbb{H})$ 
 where $\mathbb{H}$ is the quaternion division algebra over a $p$-adic field.

 \paragraph{}We recall that, for $\mu$ a character of $H$, a representation $(\pi, V)$ of $G$ is said to be $\mu$-distinguished under $H$ ($H$-distinguished if $\mu=1$ and just distinguished if 
 there is no possible confusion) if the space of $H$-homomorphisms $\mathrm{Hom}_H(\pi,\mu)$ is non-zero \textit{i.e.} if there exists a non-zero linear form $L$ on $V$ such that $L(\pi(h).v)=\mu(h)L(v)$ 
 for all $h\in H$, for all $v\in V$. \\
 Now, any character of $GL(n,D)$ can be written as $\chi\circ N_{rd,F}$ where $\chi$ is a character of $F^*$. For $\chi$ a character of $F^*$, 
 the Steinberg representation $St(n,\chi)$ of $GL(n,D)$ (denoted $St(\chi)$ if the context is clear) 
 is given by $ind_{P_\emptyset}^G(\chi\circ N_{rd,F})/\underset{P}\sum ind_P^G(\chi\circ N_{rd,F})$ where $P_\emptyset$ denotes the minimal standard parabolic subgroup of $G$ and where the standard parabolic subgroups $P$ 
 in the sum correspond to a partition of $n$ with all elements equal to $1$ except one of them which is equal to $2$. 
 
 \paragraph{}In the first part of this paper, we recall some useful elementary results to study the distinction, mainly Frobenius reciprocity and Mackey theory.\\
 
 In Sections 2 and 3, we study the distinction of the Steinberg representation according to the parity of $d$. 
 We follow the method used by Matringe in \cite{M.16} which we recall now. 
 We start by determining a set of representatives of double cosets $P\backslash G/H$ for $P$ a standard parabolic 
 subgroup of $G$. These representatives allow us to apply Mackey theory, which is, with Frobenius reciprocity and modulus characters computations, an 
 essential tool to establish a necessary condition for the distinction of the Steinberg representation. \\
First, we show that $\mathrm{Hom}_H(ind_{P_\emptyset}^G(1),\tilde{\mu})$ is at most one dimensional, hence $\mathrm{Hom}_H(St(1),\tilde{\mu})$ as well. 
It moreover implies that $\mathrm{dim}(\mathrm{Hom}_H(St(1),\tilde{\mu}))=1$ if and only if there is a nonzero $\tilde{\mu}$-equivariant linear form on $ind_{P_\emptyset}^G(1)$ 
which vanishes at each term of $\underset{P\text{ of type }(1,\dots,1,2,1,\dots,1)}\sum ind_P^G(1)$. In particular the necessary condition on $\mu$ for 
$St(1)$ to be $\tilde{\mu}$-distinguished will come from the fact that $ind_{P_\emptyset}^G(1)$ must also be $\tilde{\mu}$-distinguished in this case.
On the other hand, if $ind_{P_\emptyset}^G(1)$ is $\tilde{\mu}$-distinguished, we have an explicit $\tilde{\mu}$-equivariant linear form on its space given by an integral. To get our sufficient condition, we 
will show that this linear form does not vanish on $ind_{P_\emptyset}^G(1)$ for a well chosen standard parabolic subgroup $P$ of type $(1,\dots,1,2,1,\dots,1)$ when 
$\mu$ not of the correct form.\\

In the last section, we explain how to reformulate the Prasad and Takloo-Bighash conjecture for the Steinberg representation and we do the epsilon factor calculation in order to prove it.

\paragraph{Notation}
We let $P_\emptyset$ denote the minimal standard parabolic subgroup (of upper triangular matrices), $M_\emptyset$ its standard Levi subgroup and $N_\emptyset$ its unipotent radical. 
We set $P_\emptyset^-$ the subgroup of $G$ of lower triangular matrices. We denote by $\Phi$ the roots of the center $Z({M_{\emptyset}})$ of ${M_{\emptyset}}$ acting on the Lie algebra $Lie(G)$, by $\Phi^+$ those corresponding to the restriction 
of this action on $Lie(N_\emptyset)$, and by $\Phi^-$ those corresponding to the restriction 
of this action on $Lie(N_\emptyset^-)$. In particular $Lie(N_\emptyset)=\oplus_{\alpha \in \Phi^+} Lie(N_\alpha)$ and 
$Lie(N_\emptyset^-)=\oplus_{\alpha \in \Phi^-} Lie(N_\alpha)$, with 
obvious notation. If $P$ is a parabolic subgroup of $G$ containing $P_\emptyset$, with standard Levi factor $M$, we denote by $\Phi_M$ 
the roots of $Z({M_{\emptyset}})$ on $Lie(M)$. We define $\Phi_M^+$ and $\Phi_M^-$ in a similar fashion as above.
\paragraph{}
 For $\chi$ a character of a parabolic subgroup $P$, $ind_P^G(\chi)$ is the un-normalized parabolic induction. We let $\Delta_X$ denote the modular character of a locally compact totally disconnected topological 
 group $X$ \textit{i.e.} such that $\lambda(xg)=\Delta_X(g)\lambda(x)$ $\forall x,g\in X$, for $\lambda$ a left Haar measure on $X$. We set $\delta_X:=\Delta_X^{-1}$.
 
 \paragraph{Acknowledgements}
 I am very grateful to my supervisor Nadir Matringe for all his help and his essential explanations. I also thank Paul Broussous, my co-supervisor, for his useful explanations. 
 
 \section{General results}
 
 We gather together some useful results (written as in \cite{M.16}). 
 
 \paragraph{}We only consider smooth representations on complex vector spaces. 
Let $X$ be a locally compact totally disconnected space, and $L$ a locally compact totally disconnected group acting 
continuously and properly on $X$. If $\chi$ is a character of $L$, we denote 
by $\mathcal{C}_c^\infty(L\backslash X,\chi)$ the space of smooth functions on $X$, with compact support mod $L$, and which transform by $\chi$ under left translation by elements of $L$. If $X$ is a group $Q$ which contains $L$, then we write 
$ind_{L}^Q(\chi)$ for $\mathcal{C}_c^\infty(L\backslash Q,\chi)$, which is a representation of $Q$ by right translation.
 We shall often use the following two theorems, which are respectively Frobenius reciprocity and Mackey theory for compactly induced representations. The first one is a consequence of Proposition 2.29 of \cite{BZ.76}

\begin{prop}\label{Frob}
Let $\chi$ be a character of $L$, then the vector space $Hom_Q(ind_{L}^Q(\chi),\mu)$ is isomorphic to 
$Hom_L(\Delta\chi,\mu)$, where $\Delta$ is the quotient of the modulus character of $L$ by that of $Q$.
\end{prop}

The next one is a consequence of Theorem 5.2 of \cite{BZ.77}. 

\begin{prop}\label{Mackey}
Let $P$ be a parabolic subgroup of $G$, and $\mu$ be a character of $P$. Take a set of representatives $(u_1,\dots,u_r)$ 
of $P\backslash G/ H$, ordered such that $X_i=\coprod_{k=1}^i Pu_k H$ is open in $G$ for each $i$. Then $ind_P^G(\mu)$ is filtered by 
the $H$-submodules $\mathcal{C}_c^\infty(P\backslash X_i,\mu)$, and 
$$\mathcal{C}_c^\infty(P\backslash X_i,\mu)/\mathcal{C}_c^\infty(P\backslash X_{i-1},\mu)\simeq 
\mathcal{C}_c^\infty(P\backslash Pu_i H,\mu).$$
\end{prop}

Finally, we recall the following result from \cite{H-W.93}, which is Proposition 3.4 in there. It in particular implies that if 
$P$ contains a minimal $\tau$-split parabolic subgroup (see below), then $\mathcal{C}_c^\infty(P\backslash P H,\mu)$ is 
a subspace of $ind_P^G(\mu)$.

\begin{prop}\label{ouvert}
Let $P$ be a parabolic subgroup of $G$, and $\tau$ be an $F$-rational involution of $G$. The class $PH$ is open if and only 
if $P$ contains a minimal parabolic $\tau$-split subgroup $P'$ (which means that $\tau(P')$ and $P'$ are opposite parabolic subgroups).
\end{prop}

\section{Case $d$ even}

\subsection{Preliminaries}
We fix $F$ a non archimedean local field of characteristic not $2$ and $E$ a quadratic extension. We let $|\cdot|_F$ denote the normalized absolute value on $F$ and $N_{E/F}$ the norm map from $E^*$ to $F^*$. 
 Let $D$ be a central division $F$-algebra of dimension $d^2$ with $d=2d'$ an even positive integer. As $2$ divides $d^2$, $E$ can be seen as a subfield of $D$. Moreover, as car($F$)$\neq2$, there exists $\delta$ in $E$ with 
$\delta^2\in F\setminus F^2$ such that $E=F[\delta]$.\\
As $E$ is a subfield of $D$, we can consider the centralizer of $E$ in $D$ : $C_D(E)$. It is easy to see that $C_D(E)$ is a division $E$-algebra. The double centralizer property says that $[E:F][C_D(E):F]=[D:F]=d^2$ 
so $[C_D(E):F]=d^2/2$ and that $E=C_D(C_D(E))$ so the center of $C_D(E)$ is $E$. In summary, $D':=C_D(E)$ is a central divsion $E$-algebra of dimension $d^2/4=d'^2$. 
We recall that $N_{rd,F}$ denotes the reduced norm on $GL(n,D)$ and $N_{rd,E}$ denotes the reduced norm of $GL(n,D')$ as well as its restriction to any subgroup.
\paragraph{}It is easy to see that $C_D(E)$ is the fixed points set $D^\sigma$ of $D$ under the involution $\sigma=int(\delta)$. As $\sigma$ is an involution different of identity, $-1$ is an eigenvalue of $\sigma$ so we can chose $\iota\in D$ 
(which is thus not in $D'$) such that $\sigma(\iota)=-\iota$ and $D=D'\bigoplus \iota D'$. We can easily see that $C_{\mathcal{M}_n(D)}(E)=\mathcal{M}_n(C_D(E))=\mathcal{M}_n(D')$.\\
If we denote again by $\sigma$ the involution of $GL(n,D)$ which is given by applying $int(\delta)$ to each entry of a matrix of $GL(n,D)$, then $H=GL(n,D')$ is the fixed points of $G=GL(n,D)$ under $\sigma$~: $H=G^\sigma$.

\subsection{Representatives of $P\backslash G/H$}
Let $P$ be a standard parabolic subgroup of $G=GL(n,D)$, corresponding to a partition $\bar{n}=(n_1,\dots,n_r)$ of $n$. We denote $I(\bar{n})$ the set of symmetric matrices with natural number entries such that the sum of the $i$-th 
row equals $n_i$ for all $i$ in $\{1,\dots,r\}$. Let $\mathcal{B}=(e_1,\dots,e_n)$ be the canonical basis of $D^n$. For $s=(n_{i,j})_{1\leq i,j\leq r}\in I(\bar{n})$, we set 
$$\mathcal{B}_{i,j}=(e_{(n_1+\dots+n_{i-1}+n_{i,1}+\dots+n_{i,j-1}+1)},\dots,e_{(n_1+\dots+n_{i-1}+n_{i,1}+\dots+n_{i,j-1}+n_{i,j})})$$
We set $u_s\in GL(n,D)$ the matrix in the basis $\mathcal{B}$ of the endomorphism of $D^n$ mapping Vect$(\mathcal{B}_{i,i})_D$ to itself and Vect$(\mathcal{B}_{i,j}\cup\mathcal{B}_{j,i})_D$ to itself for $i\neq j$ such that 
$u_s$ restricted to Vect$(\mathcal{B}_{i,i})_D$ is $I_{n_{i,i}}$ and $u_s$ restricted to Vect$(\mathcal{B}_{i,j}\cup\mathcal{B}_{j,i})_D$ is $\begin{pmatrix}I_{n_{i,j}}&-\iota I_{n_{i,j}}\\I_{n_{i,j}}&\iota I_{n_{i,j}}\end{pmatrix}$.

\begin{prop}
 The set of elements $u_s$ for $s$ in $I(\bar{n})$ as described above is a set of representatives of the double cosets $P\backslash G/H$.
\end{prop}
\begin{proof}
 The proof is similar to the odd case of \cite{M.16} and with more precisions in \cite{M.11}.
\end{proof}

\begin{rem}
 If we denote $\mathbb{H}$ the division algebra of quaternions over $F$ (\textit{i.e.} of dimension $4$ over $F$), then $C_\mathbb{H}(E)=E$. Moreover, if we denote $\iota'$ an element of $\mathbb{H}\setminus E$ such that $\mathbb{H}=E\oplus \iota'E$ 
 and $\sigma(\iota')=-\iota'$ and $u_s'$ the same matrix as $u_s$ where $\iota$ have been replaced by $\iota'$, then the map $u_s\mapsto u_s'$ for $s\in I(\bar{n})$ induces a bijection between $P(D)\backslash GL(n,D)/GL(n,D')$ and 
 $P(\mathbb{H})\backslash GL(n,\mathbb{H})/GL(n,E)$.
\end{rem}

\paragraph{}Now, for $s$ in $I(\bar{n})$, we set $w_s=u_s\sigma(u_s^{-1})$. We can write $I=\llbracket 1,n\rrbracket$ as the ordered disjoint union
$$I=I_{1,1}\cup I_{1,2}\cup \dots \cup I_{1,r} \cup I_{2,1}\cup \dots \cup I_{r,1}\cup \dots \cup I_{r,r-1}\cup I_{r,r},$$ 
with $I_{i,j}$ of length $n_{i,j}$. We can check that $w_s$ is the matrix of the permutation of $n$ 
sending $I_{i,i}$ to itself identically and $I_{i,j}$ to $I_{j,i}$ when $i\neq j$ such that $w_s(n_{1,1}+\dots+n_{i,j-1}+k)=n_{1,1}+\dots+n_{j,i-1}+k$ for $k\in\{1,\dots,n_{i,j}\}$.\\
We also set 
$$\begin{matrix}&&&&\\\sigma_s&:&G=GL(n,D)&\rightarrow&GL(n,D)\\&&x&\mapsto&w_s\sigma(x)w_s^{-1}\end{matrix}.$$ 
Then, $u_sHu_s^{-1}$ is the group of fixed points of $G$ under the involution $\sigma_s$ : $u_sHu_s^{-1}=:G^{\sigma_s}$.

\paragraph{}Now, for $s\in I(\bar{n})$, we can consider the standard parabolic subgroup of $G$ attached to $s$ and its standard decomposition denoted by : $P_s=M_sN_s$. We notice that as $s$ can be seen as a sub-partition of $(n_1,\dots,n_r)$ 
(corresponding to the standard parabolic subgroup $P$), $P_s$ is included  in $P$. We need now to study $P_s\cap u_sHu_s^{-1}$ and especially its decomposition.
\paragraph{}The same proof as in Lemma 3.2 and Proposition 3.2 of \cite{M.16} shows the following proposition~: 

\begin{prop}
 For $s\in I(\bar{n})$, we have $w_s(\Phi_M^-)\subset\Phi^-$, $w_s(\Phi_M^+)\subset\Phi^+$. Moreover, $P\cap u_sHu_s^{-1}=P_s\cap u_sHu_s^{-1}$ and $P_s\cap u_sHu_s^{-1}$ is the semidirect product of $M_s\cap u_sHu_s^{-1}$ and $N_s\cap u_sHu_s^{-1}$.
\end{prop}

We can make the subgroup $M_s^{\sigma_s}=M_s\cap u_sHu_s^{-1}$ explicit : 
$$M_s^{\sigma_s}=\{diag(a_{1,1},a_{1,2},\dots,a_{1,r},a_{2,1},\dots,a_{r,r})\;;\; a_{j,i}=\sigma(a_{i,j})\in GL(n_{i,j},D)\}.$$

\paragraph{}Finally, we have the following equality of characters :
\begin{prop}\label{car2}
 $(\delta_{P_s^{\sigma_s}})_{|M_s^{\sigma_s}}=(\delta_{P_s})_{|M_s^{\sigma_s}}$
\end{prop}
\begin{proof}
As said in Proposition 4.4 of \cite{M.11}, thanks to Lemma 1.10 of \cite{K-T.08}, as $\sigma_s$ is defined over $E$, it is enough to check the equality on the $E$-split component $Z_s^{\sigma_s}$ (the maximal $E$-split 
 torus in the center of $M_s$) of $M_s^{\sigma_s}$. Here, 
 $$Z_s^{\sigma_s}=\left\{\begin{pmatrix}\lambda_{1,1}I_{n_{1,1}}&&&\\&\lambda_{1,2}I_{n_{1,2}}&&(0)\\(0)&&\ddots&\\&&&\lambda_{r,r}I_{n_{r,r}}\end{pmatrix};\;\lambda_{i,j}=\lambda_{j,i}\in E^*\right\}$$
 For $t\in Z_s^{\sigma_s}$, $\delta_{P_s}(t)=|N_{rd,F}(Ad(t)_{|Lie(N_s)})|_F=\prod_{\alpha\in\Phi^+-\Phi_s^+}|N_{rd,F}(\alpha(t))|_F$.
 \paragraph{}Now, for $\alpha\in\Phi$, let $\mathcal{N}_{\alpha,w_s(\alpha)}=\{x\in Lie(N_\alpha)+Lie(N_{w_s(\alpha)});\;\sigma_s(x)=x\}$; it is a right $D'$-vector space of dimension $|\alpha,w_s(\alpha)|$. Then, for $t\in Z_s^{\sigma_s}$, we have :
 $$\delta_{P_s^{\sigma_s}}(t)=\prod_{\{\alpha,w_s(\alpha)\}\subset\Phi^+-\Phi_s^+}|N_{rd,E}(Ad(t)_{|\mathcal{N}_{\alpha,w_s(\alpha)}})|_E=\prod_{\{\alpha\in\Phi^+-\Phi_s^+;\;w_s(\alpha)\in\Phi^+-\Phi_s^+\}}|N_{rd,E}(\alpha(t))|_E.$$
 
 Moreover, we have (see Proposition 4.4 of \cite{M.11}) : 
\begin{equation}\label{relation2}
\prod_{\{\alpha\in \Phi^+-\Phi_s^+; w_s(\alpha)\notin \Phi^+-\Phi_s^+\}}|N_{rd,E}(\alpha(t))|_E= \prod_{\{\alpha\in \Phi^+-\Phi_s^+; w_s(\alpha)\in \Phi^--\Phi_s^-\}}|N_{rd,E}(\alpha(t))|_E= 1
\end{equation}
 
 \paragraph{}Thus, $\delta_{P_s^{\sigma_s}}(t)=\underset{\{\alpha\in\Phi^+-\Phi_s^+\}}\prod|N_{rd,E}(\alpha(t))|_E$.
 \paragraph{}Finally, as $\alpha(t)\in D'$, we have 
 $$|N_{rd,F}(\alpha(t))|_F=|N_{E/F}\circ N_{rd,E}(\alpha(t))|_F=|N_{E/F}\circ N_{rd,E}(\alpha(t))|_E^{1/2}=|N_{rd,E}(\alpha(t))|_E$$ 
 which gives the equality of characters.
\end{proof}

\subsection{Distinguished Steinberg representations}
For $\mu$ a character of $E^*$, we set $\tilde{\mu}=\mu\circ N_{rd,E}$. In this section, we will study whether $St(1)$ is $\tilde{\mu}$-distinguished under $H$ or not, according to the character $\mu$ of $E^*$. We denote by $St(1)$ the Steinberg representation 
$ind_{P_\emptyset}^G(1)/\sum_P ind_P^G(1)$ 
where $P$ describes the standard parabolic subgroups of $G$ corresponding to a partition of $n$ of type $(1,\dots,1,2,1,\dots,1)$.

\paragraph{}First, we suppose that $St(1)$ is $\tilde{\mu}$-distinguished under $H$ and we find a necessary condition on $\mu$ in the following proposition :
\begin{prop}\label{necessaire2}
 Suppose that $St(1)$ is $\tilde{\mu}$-distinguished under $H$. Then, $\mu_{|F^*}=1$ if $n$ is even and $\mu=1$ if $n$ is odd. Moreover, only the open orbit $P_\emptyset u_sH$ where $s=\begin{pmatrix}&&1\\&\iddots&\\1&&\end{pmatrix}$ 
 supports a $\tilde{\mu}$-equivariant linear form and dim$(Hom_H(St(1),\tilde{\mu}))=1$.
\end{prop}
\begin{proof}
The idea of the proof is the same as those of Propositions 3.4 and 3.5 of \cite{M.16} so we do not give all details. 
Supppose that $St(1)$ is $\tilde{\mu}$-distinguished, then $ind_{P_\emptyset}^G(1)$ is also $\tilde{\mu}$-distinguished so
\begin{eqnarray*}
\exists s\in I(\bar{n})\text{ such that }\mathrm{Hom}_{P_\emptyset^{\sigma_s}}\left(\frac{\Delta_{P_\emptyset^{\sigma_s}}}{\Delta_{G^{\sigma_s}}}\tilde{\mu}_s^{-1},1\right)\neq\{0\}
\end{eqnarray*}
where $\tilde{\mu}_s(x)=\tilde{\mu}(u_s^{-1}xu_s)$ for $x$ in $u_s H u_s^{-1}$.
\paragraph{}\begin{sloppypar}Now, suppose that $w_s$ has at least one fixed point, then if we consider $M=diag(1,\dots,1,a,1,\dots,1)$ with $a\in F^*$ in the $i$-th row, the previous equality of characters and the fact that $\tilde{\mu}_{|F^*}$ is unitary (considering the central character of 
$ind_{P_\emptyset}^G(1)$ and its $\tilde{\mu}$-distinction) imply that 
 $i=\frac{n+1}{2}$ so $n$ is odd and $w_s$ has only one fixed point.
 \end{sloppypar}
 
 \paragraph{}Thus, if $n$ is even, $w_s$ has no fixed point. Then we get  $s=\begin{pmatrix}&&1\\&\iddots&\\1&&\end{pmatrix}$ and $\tilde{\mu}(u_s^{-1}M_{\emptyset}^{\sigma_s}u_s)=1$. 
 Let us prove that $N_{rd,E}(u_s^{-1}M_\emptyset^{\sigma_s}u_s)=F^*$. First, we show how $D$ can be embedded in $\mathcal{M}_2(D')$. $\mathcal{M}_2(D')$ identifies to $\mathrm{End}(D)_{D'}$ via the basis $(1,\iota)$ of the right 
 $D'$-vector space $D$. Moreover, we have 
 $$\begin{matrix}D\otimes_FE&\tilde{\longrightarrow}&\mathrm{End}(D)_{D'}\\a\otimes e&\mapsto&(d\mapsto ade)\end{matrix}$$
 so $D$ can be embedded in $\mathcal{M}_2(D')$ via 
 $$\begin{matrix}f\::\:&D&\longrightarrow&\mathcal{M}_2(D')\\&a&\mapsto&a\otimes 1\in D\otimes_FE\simeq\mathcal{M}_2(D')\end{matrix}$$
 Thus, $N_{rd,E}(f(D^*))=N_{rd,F}(D^*)=F^*$ (a splitting field $L$ for $D$ is also a splitting field for $D\otimes_FE$ and $(D\otimes_FE)\otimes_EL=D\otimes_FL$). We finish by noticing that $u_s^{-1}M_\emptyset^{\sigma_s}u_s=f(D)$. 
 Thus, $N_{rd,E}(u_s^{-1}M_\emptyset^{\sigma_s}u_s)=F^*$ so we obtain $\mu_{|F^*}=1$.\\
 If $n$ is odd, we get $s=\begin{pmatrix}&&1\\&\iddots&\\1&&\end{pmatrix}$, $\mu_{|F^*}=1$ and, as $w_s$ has one fixed point, we have the extra condition $\tilde{\mu}(D'^*)=1$. As $N_{rd,E}(D'^*)=E^*$, we obtain $\mu=1$.
 
 \paragraph{}We notice that, as $\sigma(u_s^{-1}P_\emptyset u_s)=u_s^{-1}P_\emptyset^- u_s$, then $u_s^{-1}P_\emptyset u_sH$ is open in $G$ (thanks to Proposition \ref{ouvert}) so $P_\emptyset u_sH$ is open in $G$ too.
\end{proof}

Now, we will exhibit a non-zero $\tilde{\mu}$-equivariant linear form on $ind_{P_\emptyset}^G(1)$. To do that, we will follow the strategy of Sections 4.2 and 4.3 of \cite{M.16}.

\paragraph{}For $z\in\mathbb{C}$, we denote by $\delta_z$ the character $(\delta_{P_\emptyset})^z$. For $f\in ind_{P_\emptyset}^G(1)$, we denote by $f_z$ the only element in $ind_{P_\emptyset}^G(\delta_z)$ such that 
$f_{z|K}=f_{|K}$ with $K=GL(n,\mathcal{O}_D)$. The map 
$$\phi : \begin{matrix}&&\\ind_{P_\emptyset}^G(1)&\rightarrow&ind_{P_\emptyset}^G(\delta_z)\\f&\mapsto&f_z\end{matrix}$$ 
is a $K$-isomorphism. We also set $s_0=\begin{pmatrix}&&1\\&\iddots&\\1&&\end{pmatrix}\in I(\bar{n})$, 
$u_0=u_{s_0}$ and $\sigma_0=\sigma_{s_0}$.

\begin{prop}
 Suppose that $\mu_{|F^*}=1$ if $n$ is even and $\mu=1$ if $n$ is odd. For $f$ in $ind_{P_\emptyset}^G(1)$, the integral $I_{n,z}(f_z)=\int_{u_0^{-1}P_\emptyset u_0\cap H\backslash H}\tilde{\mu}^{-1}(h)f_z(u_0h)dh$ converges for $Re(z)$ large enough and 
 there exists $Q\in\mathbb{C}[X]$ such that $Q(q^{-z})I_{n,z}(f_z)$ 
 belongs to $\mathbb{C}[q^{\pm z}]$ for all $f$ in $ind_{P_\emptyset}^G(1)$ (with $q=|k_F|$).\\
 Moreover, for $z'$ in $\mathbb{C}$, there exists $l_{z'}$ in $\mathbb{N}$ such that 
 $$L_{z'}=\underset{z\rightarrow z'}{lim}(1-q^{z'-z})^{l_{z'}}I_{n,z}$$
 belongs to $\mathrm{Hom}_H(ind_{P_\emptyset}^G(\delta_{z'}),\tilde{\mu})\backslash\{0\}$.
\end{prop}
\begin{proof}
 This is due to the Bernstein principle for meromorphic continuation of equivariant linear forms (see Corollary 2.12 in \cite{M.15} which is stated for $GL(n,F)$ but is also true for $GL(n,D)$). First, notice that 
 $$P_\emptyset^{\sigma_0}=P_\emptyset\cap u_0Hu_0^{-1}=\{diag(a_1,\dots,a_{n/2},\sigma(a_{n/2}),\dots,\sigma(a_1)); a_i\in D^*\}\text{ if $n$ is even and }$$
 $$P_\emptyset^{\sigma_0}=\{diag(a_1,\dots,a_{\frac{n-1}{2}},b,\sigma(a_{\frac{n-1}{2}}),\dots,\sigma(a_1)); a_i\in D^*, b\in D^{'*}\}\text{ if $n$ is odd.}\hspace{1.5cm}$$
 Thus, as $N_{rd,F}(a)=N_{rd,F}(\sigma(a))$ for $a\in D^*$, then $\delta_z(u_0^{-1}P_\emptyset u_0\cap H)=1$ and $\tilde{\mu}(u_0^{-1}P_\emptyset u_0\cap H)=1$ so $I_{n,z}(\mathcal{C}_c^\infty(P_\emptyset\backslash P_\emptyset u_0 H,\delta_{z}))$ is 
 non-zero and well defined for all $z$ in $\mathbb{C}$. As $\mathcal{C}_c^\infty(P_\emptyset\backslash P_\emptyset u_0 H,\delta_{z})\subset ind_{P_\emptyset}^G(\delta_{z})$ (because $P_\emptyset u_0 H$ is open in $G$), it means that for all $z$ in $\mathbb{C}$, 
 there exists $\tilde{f}_{z}$ in $ind_{P_\emptyset}^G(\delta_{z})$ such that $I_{n,z}(\tilde{f}_{z})\neq0$.
 
\paragraph{}To see that $I_{n,z}(f_z)$ is absolutely convergent for $Re(z)$ large enough, first we notice that $|\tilde{\mu}^{-1}|=|\mu^{-1}|\circ N_{rd,E}$ can be written $|\tilde{\mu}^{-1}|=|\cdot|_E^\alpha\circ N_{rd,E}$ for $\alpha$ a complex number. As $|\cdot|_E=|\cdot|_F\circ N_{E/F}$ and as $N_{rd,F|H}=N_{E/F}\circ N_{rd,E|H}$, then $|\tilde{\mu}^{-1}|=|\cdot|_F^\alpha\circ N_{rd,F|H}$ so $|\tilde{\mu}^{-1}|$ can be extended to a character of $G$ that we denote by $\chi$. Now, we have to prove that   $\int_{u_0^{-1}P_\emptyset u_0\cap H\backslash H}\chi(h)f_z(u_0h)dh$ is absolutely convergent. We can see $f_z\mapsto \int_{u_0^{-1}P_\emptyset u_0\cap H\backslash H}\chi(h)f_z(u_0h)dh$ as a function of $\chi f_z$ which belongs to $\chi\otimes ind_{P_\emptyset}^G(\delta_z)$. Then, the absolute convergence comes from Theorems 2.8 and 2.16 of \cite{B-D.08}.

 \paragraph{}Finally, we can use the Bernstein principle for meromorphic continuation of equivariant linear forms, because the space $\mathrm{Hom}_H(ind_{P_\emptyset}^G(\delta_{z}),\tilde{\mu})$ is of dimension $\leq1$ for all $q^{-s}$. Indeed, we prove it 
 as in the beginning of the proof of Proposition \ref{necessaire2}. We suppose that $ind_{P_\emptyset}^G(\delta_{z})$ is $\tilde{\mu}$-distinguished under $H$. It implies that $\tilde{\mu}_{|F^*}$ is unitary (considering the central character). Then, we apply 
 Mackey theory and Frobenius reciprocity and we get that there exists a unique $s$ (it is $s_0$) such that $\mathrm{Hom}_{P_\emptyset^{\sigma_s}}(\Delta_{P_\emptyset^{\sigma_s}}\delta_z,\tilde{\mu}_s)$ is non-trivial. Thus, we obtain that 
 $\mathrm{Hom}_{P_\emptyset^{\sigma_s}}(\Delta_{P_\emptyset^{\sigma_s}}\delta_z,\tilde{\mu}_s)$ is of dimension $1$ so $\mathrm{Hom}_H(ind_{P_\emptyset}^G(\delta_{z}),\tilde{\mu})$ is $1$-dimensional too.
 \end{proof}

\paragraph{}Now, we can give two results of distinction, according to the parity of $n$ : 

\begin{prop}\label{dist}
 If $n$ is odd and if $\mu=1$, then $St(1)$ is $\tilde{\mu}$-distinguished.
\end{prop}
\begin{proof} 
 Taking $z'=0$, we get that $L_0$ is a non-zero $\tilde{\mu}$-equivariant linear form on $ind_{P_\emptyset}^G(1)$ which is thus $\tilde{\mu}$-distinguished. We end this proof as in Proposition 3.6 of \cite{M.16}.

\end{proof}

\begin{prop}
 If $n$ is even and if $\mu_{|F^*}=1$ and $\mu\neq1$, then $St(1)$ is $\tilde{\mu}$-distinguished.
\end{prop}
\begin{proof}
The proof is the same as in Proposition \ref{dist}.
\end{proof}

To finish this case, we have the following proposition :

\begin{prop}
 If $n$ is even and if $\mu=1$, then $St(1)$ is not $\tilde{\mu}$-distinguished.
\end{prop}
To prove this theorem, we need first two lemmas about the integral $I_{n,z}$. We denote by $\Phi$ the constant function equal to $1$ in $ind_{P_\emptyset}^G(1)$.  Then, for any $z\in\mathbb{C}$, $f_z=f\Phi_z$. If $n=2$, we denote $\Phi$ by $\Phi_2$.

\begin{lem}
Suppose that $\mu=1$ and that $n=2$, then up to a unit in $\mathbb{C}[q^{\pm z}]$, we have :
$$I_{2,z}(\Phi_{2,z})=L(dz-\frac{d}{2},1_{E^*})L(dz,1_{E^*})/L(2dz,1_{F^*}),$$
where $L$ is the usual Tate $L$-factor. In particular, $I_{2,0}(\Phi_2)\neq0$.
\end{lem}
\begin{proof}
The proof is the same as Proposition 4.5 of \cite{M.16} with $u_0=\begin{pmatrix}1&-\iota\\1&\iota\end{pmatrix}$. We only give the beginning and the end. We set $\nu_F:=|N_{rd,F}(\cdot)|_F$.
For $Re(z)$ large enough, we have : 
$$\Phi_{2,z}(g)=\nu_F(g)^{dz}\int_{D^*}\Phi((0,t)g)\nu_F(t)^{2dz}dt/L(2dz,1_{F^*}).$$
\paragraph{}Finally, if we set $\epsilon=(\nu_F(h_\iota^{-1})\nu_F(u_0))^{dz}$ and denote by $\Phi_0$ the characteristic function of $\mathcal{M}(2,\mathcal{O}_{D'})$ (and recalling that $d'=\frac{d}{2}$ is the index of $D'$ over $E$), we obtain :
\begin{eqnarray*}
 I_{2,z}(\Phi_{2,z})&=&\epsilon\int_H\Phi_0(h)\nu_F(h)^{dz}dh/L(2dz,1_{F^*})\\
 &=&\epsilon\int_H\Phi_0(h)\nu_E(h)^{2d'z}dh/L(2dz,1_{F^*})\\
 &=&\epsilon L(2d'z-\frac{1}{2}(2d'-1),1_H)/L(2dz,1_{F^*})\\
 &&\text{ by definition of $Z$ and $L$-functions}\\
 &=&\epsilon L(2d'z-\frac{1}{2}(3d'-1),1_{D'^*})L(2d'z-\frac{1}{2}(d'-1),1_{D'^*})/L(2dz,1_{F^*})\\
 &&\text{by inductivity relation of the Godement-Jacquet $L$-factor }L(z,1_H)\\
 &=&\epsilon L(2d'z-d',1_{E^*})L(2d'z,1_{E^*})/L(2dz,1_{F^*})
\end{eqnarray*}
If $z=0$, $L(2d'z,1_{E^*})$ and $L(2dz,1_{F^*})$ have one simple pole and $L(2d'z-d',1_{E^*})$ has no pole so $I_{2,0}(\Phi_2)\neq0$.
\end{proof}

We recall Proposition 4.6 of \cite{M.16} :
\begin{lem}\label{nonzero}
Suppose that $\mu=1$. For $n=2m$, let $P$ be the standard parabolic subgroup of $G$ corresponding to the partition $\bar{n}=(1,\dots,1,2,1,\dots,1)$ with $n_{m}=2$. Then, there is $f$ in $ind_P^G(1)$ such that $I_{n,z}(f\Phi_z)=I_{2,z}(\Phi_{2,z})$. In particular, taking $z=0$, one has $I_{n,0}(f)=I_{2,0}(\Phi_2)\neq0$.
\end{lem}
\begin{proof}
The proof is exactly the same as in Proposition 4.6 of \cite{M.16}. 
\end{proof}
Finally, we come back to the proof of the proposition :
\begin{proof}[Proof of the proposition]
 Suppose that $\mu=1$ and $St(1)$ is $\tilde{\mu}$-distinguished (\textit{i.e.} $H$-distinguished). Then, $ind_{P_\emptyset}^G(1)$ is $H$-distinguished so there exists $L$ a non-zero $H$-invariant linear form on $ind_{P_\emptyset}^G(1)$. 
 As $\mathrm{dim}(\mathrm{Hom}_H(ind_{P_\emptyset}^G(1),1))=1$ (thanks to Proposition \ref{necessaire2}), then $L$ equals to $L_0$ up to a non-zero scalar. As $St(1)$ is distinguished, $L_{|ind_P^G(1)}$ must be equal to zero 
 for all standard parabolic subgroups $P$ of type $(1,\dots,1,2,1,\dots,1)$. Moreover, as said in Proposition \ref{necessaire2}, $L_0$ restricts non trivially to 
 $\mathcal{C}_c^\infty(P_\emptyset\backslash P_\emptyset u_0H,1)\subset ind_{P_\emptyset}^G(1)$.\\
 As $I_{n,0}(\mathcal{C}_c^\infty(P_\emptyset\backslash P_\emptyset u_0H,1))$ is non zero (and is well defined), this implies that $L_0=I_{n,0}$. Now, we take $f$ and $P$ as in Lemma \ref{nonzero}. Then we have 
 $L_0(f)=I_{n,0}(f)\neq0$ which contradicts the distinction of $St(1)$.
\end{proof}

\section{Case $d$ odd and $n$ even}
 
 \subsection{Preliminaries}
 
 We set $n=2m$ and we suppose that $d$ (the index of $D$ over its center $F$) is odd. Let us consider $D\otimes_F E$ which is a central division $E$-algebra of dimension $d^2$ (thanks 
 to Wedderburn structure theorem and Hasse's invariant). We can choose $\delta\in (D\otimes_F E)\backslash D$ such that $\Delta:=\delta^2\in F$ (for example, $\delta=1\otimes x$ with $x$ in $E\backslash F$ 
 such that $x^2$ is in $F$). Then, $D\otimes_F E$ is of dimension $2$ over $D$ so
 we can write $D\otimes_F E$ as $D\oplus \delta D$ and $D\otimes_FE$ identifies as 
 a right $D$-vector space of dimension $2$.\\
 Now, if we let $(e_1,\dots,e_m)$ denote the canonical basis of $(D\otimes_FE)^m$, then the right $D$-vector space $(D\otimes_FE)^m$ identifies with the $D$-vector space $D^{2m}$ via the basis $\mathcal{B}=(e_1,\dots,e_m,\delta e_m,\dots,\delta e_1)$ of $D^{2m}$ so 
 $\mathrm{End}((D\otimes_FE)^m)_D \simeq \mathrm{End}(D^{2m})_D$. Now, if $u\in \mathrm{End}((D\otimes_FE)^m)_D$, it is easy to see that $u\in \mathrm{End}((D\otimes_FE)^m)_{D\otimes_FE}$ if and only if $u$ commutes with the multiplication by $\delta$ (denoted $\mu_\delta$). In the basis $\mathcal{B}$, the matrix of 
 the endomorphism $\mu_\delta$ is given by 
 $$\begin{pmatrix}&&&&&\Delta\\&&&&\iddots&\\&&&\Delta&&\\&&1&&&\\&\iddots&&&&\\1&&&&& \end{pmatrix}=:U_{\Delta_{2m}}$$
 so $\mathcal{M}_m(D\otimes_FE)$ can be viewed as the fixed points of $\mathcal{M}_{2m}(D)$ under the involution $int_{U_{\Delta_{2m}}}$ of $\mathcal{M}_{2m}(D)$ (where $int_{U_{\Delta_{2m}}}(x) = U_{\Delta_{2m}} x {U_{\Delta_{2m}}}^{-1}$). \\
 Now, $E$ can be embedded in $\mathcal{M}_{2m}(D)$ via 
 $$x\in E\mapsto \begin{pmatrix}1\otimes x&&\\&\ddots&\\&&1\otimes x\end{pmatrix}\in \mathcal{M}_m(D\otimes_FE)\subset\mathcal{M}_{2m}(D)$$
 and it is easy to check that $g\in C_{\mathcal{M}_{2m}(D)}(E)$ if and only if $g$ commutes with $U_{\Delta_{2m}}$ \textit{i.e.} if and only if $g\in\mathcal{M}_m(D\otimes_FE)$. Thus $C_{\mathcal{M}_{2m}(D)}(E)=\mathcal{M}_m(D\otimes_FE)$ 
 and we set $G=GL(2m,D)$, $H=GL(m,D\otimes_FE)=G^\sigma$ with $\sigma = int_{U_{\Delta_{2m}}}$.  
 We recall that $N_{rd,F}$ denotes the reduced norm on $GL(n,D)$ and $N_{rd,E}$ denotes the reduced norm on $GL(m,D\otimes_FE)$ (as well as its restriction to any subgroup).

 \subsection{Representatives of $P\backslash G/H$}
 
 Let $P$ be a standard parabolic subgroup of $G=GL(2m,D)$ corresponding to a partition $\bar{n}=(n_1,\dots,n_r)$ of $n=2m$. We define $I(\bar{n})$ to be the set of symmetric matrices $s=(n_{i,j})\in\mathcal{M}_r(\mathbb{N})$ with positive integral 
 entries, even on the diagonal, and such that the sum of the $i$-th row is equal to $n_i$ for all $i$ in $\{1,\dots,r\}$.
 \paragraph{}As each $n_{i,i}$ is even for $i\in\mathbb\llbracket 1,r\rrbracket$, we can write $n_{i,i}=2m_{i,i}$ and we can write $n$ as an ordered sum of integers in two different ways :
  \begin{equation*}n=m_{1,1}+n_{1,2}+\dots+n_{1,r}+m_{2,2}+n_{2,3}+\dots+n_{2,r}+m_{3,3}+\dots+m_{r-1,r-1}+n_{r-1,r}+m_{r,r}\end{equation*}
\begin{equation} +m_{r,r}+n_{r,r-1}+m_{r-1,r-1}+\dots+m_{3,3}+n_{r,2}+\dots+n_{3,2}+m_{2,2}+n_{r,1}+\dots+n_{2,1}+m_{1,1}\end{equation}
\begin{flushright}(1\textsuperscript{st} ordering)\end{flushright}
 
 \begin{equation}n=n_{1,1}+\dots+n_{1,r}+n_{2,1}+\dots+n_{2,r}+\dots+n_{r,1}+\dots+n_{r,r}\hspace{1cm}\end{equation}
 \begin{flushright}(2\textsuperscript{nd} ordering corresponding to the lexicographical ordering)\end{flushright}

We denote by $w_s$ the matrix of the permutation (still denoted $w_s$) defined as follows:\\
If $i\in \llbracket{1,r \rrbracket}$, then for $k\in \llbracket{1,m_{i,i}\rrbracket}$, we set
$$w_s(m_{1,1}+\dots+m_{i-1,i-1}+n_{i-1,i}+\dots+n_{i-1,r}+k)=n_{1,1}+\dots+n_{i-1,1}+\dots+n_{i,i-1}+k,$$ and 
$$w_s(m_{1,1}+\dots+m_{i+1,i+1}+n_{r,i}+\dots+n_{i+1,i}+k)=n_{1,1}+\dots+n_{i-1,1}+\dots+n_{i,i-1}+k+m_{i,i}.$$
If $i<j$, for $k\in \llbracket{1,n_{i,j}\rrbracket}$, we set 
$$w_s(m_{1,1}+ \dots+ m_{i,i}+ n_{i,i+1}+ \dots +n_{i,j-1}+k)=n_{1,1}+ \dots + n_{1,r}+\dots+ n_{i,1}+ \dots+ \dots +n_{i,j-1}+k $$ and 
$$w_s(m_{1,1}+ \dots+ m_{i+1,i+1}+ n_{r,i} +\dots+ n_{j+1,i}+k)= n_{1,1}+ \dots + n_{1,r}+\dots+ n_{j,1}+ \dots+ \dots +n_{j,i-1}+k .$$ 

In other words, $w_s$ sends an integer of rank $k$ according to the 1\textsuperscript{st} ordering to the integer of rank $k$ corresponding to the 2\textsuperscript{nd} ordering.

\paragraph{}A proof similar to Proposition 3.1 of \cite{M.16} shows the following result :

\begin{prop}
Let $\bar{n}$ be a partition of $n$ and $P$ be a standard parabolic subgroup of $G$ corresponding to this partition, then $G=\bigsqcup_{s\in I(\bar{n})}Pw_sH$.
\end{prop}

\begin{rem}
 \begin{sloppypar}
There is a bijection between $P_n(D)\,\backslash \,GL(n,D)\,/\,GL(m,D\otimes_F E)$ and $P_n(F)\backslash GL(n,F)/GL(m,E)$ via the identity map of $\{w_s|s\in I(\bar{n})\}$.
 \end{sloppypar}
\end{rem}

Now, for $s\in I(\bar{n})$, we set $t_s=w_sU_{\Delta_{2m}}w_s^{-1}$. It is a monomial matrix (so it is in $N_G(M_\emptyset)$) and if we let $\tau_s$ denote the image of $t_s$ in $\mathcal{\sigma}_n=N_G(M_\emptyset)/M_\emptyset$, then 
$\tau_s$ is a permutation matrix of order $2$, given by the formula $\tau_s=w_sww_s^{-1}$ (where $w=\begin{pmatrix}&&1\\&\iddots&\\1&&\end{pmatrix}$).\\
Then, we can see that $w_sHw_s^{-1}$ is the group of the fixed points of $G$ under the involution : 
$$\sigma_s : x\mapsto t_sxt_s^{-1}.$$

\paragraph{}We need to know exactly how acts the permutation $\tau_s$. One checks that $\tau_s$ is the involution of $\llbracket 1,n\rrbracket=I=I_{1,1}\cup I_{1,2}\cup \dots \cup I_{1,r} \cup \dots \cup I_{r,1}\cup \dots \cup I_{r,r-1}\cup I_{r,r}$ 
(with $I_{i,j}$ of length $n_{i,j}$), which stabilises each $I_{i,i}$, acting on it as the symmetry with respect to its midpoint, and which stabilises $I_{i,j}\cup I_{j,i}$ (for $i<j$) and acts 
on this union of intervals as the symmetry with center the midpoint of the interval joining the left end of $I_{i,j}$ and the right end of $I_{j,i}$.

\paragraph{}For $s\in I(\bar{n})$, we denote by $P_s$ the standard parabolic subgroup of $G$ corresponding to the sub-partition $s$ of $\bar{n}$. As usual, we denote $P=MN$ and $P_s=M_sN_s$ the standard Levi decomposition of $P$ and $P_s$. 
Then, again as in the even case, we have the following proposition :

\begin{lem}\label{lemme}
For $s\in I(\bar{n})$, one has $\tau_s(\Phi_M^-)\subset \Phi^-$, $\tau_s(\Phi_M^+)\subset \Phi^+$.
\end{lem}

\begin{prop}\label{decomp}
For any $s\in I(\bar{n})$, one has $P\cap w_s H w_s^{-1}=P_s\cap w_s H w_s^{-1}$, and 
$P_s\cap w_s H w_s^{-1}$ is the semidirect product of $M_s\cap w_s H w_s^{-1}$ and $N_s\cap w_s H w_s^{-1}$.
\end{prop}

We will now let $P_s^{\sigma_s}$ denote $P_s\cap w_s H w_s^{-1}$ and $M_s^{\sigma_s}$ denote $M_s\cap w_s H w_s^{-1}$. We can explicitly describe the group $M_s^{\sigma_s}$ : an element $a\in M_s^{\sigma_s}$ is of the form :
$$a=diag(a_{1,1},a_{1,2},\dots,a_{1,r},a_{2,1},\dots,a_{r,r-1},a_{r,r})$$

where $a_{i,i}\in GL(n_{i,i},D)$ satisfies $a_{i,i}=U_{\Delta}a_{i,i}U_{\Delta}^{-1}$ ($U_{\Delta}:=U_{\Delta_{n_{i,i}}}$) 
\textit{i.e.} $a_{i,i}\in GL(m_{i,i},D\otimes_FE)$ and $a_{i,j}\in GL(n_{i,j},D)$ satisfies $a_{i,j}=wa_{j,i}w^{-1}$ if $i\neq j$ ($w=\begin{pmatrix}&&1\\&\iddots&\\1&&\end{pmatrix}$) for all $i$, $j$ in $\{1,\dots,r\}$.

\begin{prop}\label{car}
We have the following equality of characters :
$$(\delta_{P_s^{\sigma_s}})_{|M_s^{\sigma_s}}=(\delta_{P_s}^{1/2})_{|M_s^{\sigma_s}} .$$
\end{prop}

\begin{proof}
Again, the method is the same as in Proposition 4.4 of \cite{M.11}.
Thanks to Lemma 1.10 of \cite{K-T.08}, it is enough to check the equality on the $F$-split component $Z_s^{\sigma_s}$ (the maximal $F$-split torus in the center of $M_s$) of $M_s^{\sigma_s}$ because $\sigma_s$ is defined over $F$. 
Thus, we will consider each character on $Z_s^{\sigma_s}$, that is to say, the subgroup of the matrices of the form 
$$\begin{pmatrix}\lambda_{1,1}I_{n_{1,1}}&&&\\&\lambda_{1,2}I_{n_{1,2}}&&\\&&\ddots&\\&&&\lambda_{r,r}I_{n_{r,r}}\end{pmatrix}$$
with $\lambda_{i,j}=\lambda_{j,i}\in F^*$ and $n_{i,i}$ even.
\paragraph{}For $\alpha\in\Phi$, we set $\mathcal{N}_{\alpha,\tau_s(\alpha)}=\{x\in Lie(N_\alpha)+Lie(N_{\tau_s(\alpha)}\: ;\:\sigma_s(x)=x\}$. It's a $D$-vector space of dimension $1$.\\
For $t\in Z_s^{\sigma_s}$, we have : 

\begin{equation}\label{relation1}
\delta_{P_s^{\sigma_s}}(t) = \prod_{\{\alpha,\tau_s(\alpha)\}\subset \Phi^+-\Phi_s^+}|N_{rd,F}(Ad(t)_{|\mathcal{N}_{\alpha,\tau_s(\alpha)}})|_F = \prod_{\{\alpha\in \Phi^+-\Phi_s^+; \tau_s(\alpha)\in \Phi^+-\Phi_s^+\}}|N_{rd,F}(\alpha(t))|_F^{1/2}
\end{equation}
\begin{sloppypar}
The second equality of \eqref{relation1} comes from the fact that if 
$\{\alpha_0,\tau_s(\alpha_0)\}\subset \Phi^+-\Phi_s^+$, then $|N_{rd,F}(Ad(t)_{|\mathcal{N}_{\alpha_0,\tau_s(\alpha_0)}})|_F =|N_{rd,F}(\alpha_0(t))|_F$ and the power $1/2$ comes from the fact that $\tau_s$ has no fixed point 
whereas $\mathcal{N}_{\alpha,\tau_s(\alpha)}$ is of dimension $1$.
\end{sloppypar}

As in the even case, $\underset{{\{\alpha\in\Phi^+-\Phi_s^+;\;\tau_s(\alpha)\notin\Phi^+-\Phi_s^+\}}}\prod|N_{rd,F}(\alpha(t))|_F=1$. This implies that 
$$\delta_{P_s}^{\sigma_s}(t)=\underset{\{\alpha\in\Phi^+-\Phi_s^+\}}{\prod}|N_{rd,F}(\alpha(t))|_F^{1/2}$$

\paragraph{}Finally, by definition we have : 
$$\delta_{P_s}(t)=|N_{rd,F}(Ad(t)_{|Lie(N_s)})|_F = \prod_{\{\alpha\in\Phi^+-\Phi_s^+\}}|N_{rd,F}(\alpha(t))|_F.$$
and we have the characters equality.

\end{proof}

\subsection{Distinguished Steinberg representations}
In this part, we will study whether the Steinberg representation is $\mu\circ N_{rd,E}$-distinguished under $H$ or not according to the character of $E^*$ considered $\mu$.
For $\mu$ a character of $E^*$, we set $\tilde{\mu}:=\mu\circ N_{rd,E}$. We recall that $St(1)$ is the Steinberg representation 
$ind_{P_\emptyset}^G(1)/S$ with $S=\sum_P ind_P^G(1)$ where the standard parabolic subgroups $P$ in the sum correspond to a partition $\bar{n}$ of $n$ with all $n_i$'s equal to $1$ except one which 
is $2$.
\paragraph{}First, we give a necessary condition on $\mu$ to allow $St(1)$ to be $\tilde{\mu}$-distinguished.

\begin{prop}\label{necessaire}
If $St(1)$ is $\tilde{\mu}$-distinguished under $H$, then $\mu_{|F^*}=1$. Moreover, only the open orbit $P_\emptyset H$ supports a $\tilde{\mu}$-equivariant linear form and $dim\left(Hom_H(St(1),\tilde{\mu})\right)= 1$.
\end{prop}

\begin{proof}
 As the method is the same as in the proof of Proposition \ref{necessaire2}, we will only underline the most important points.\\
 First, we recall that $\tilde{\mu}_{|F^*}$ is unitary. Then, $St(1)$ being $\tilde{\mu}$-distinguished implies, by Frobenius reciprocity, Mackey theory and Proposition \ref{car2}, that there 
 exists $s$ in $I(\bar{n})$ such that $(\delta_{P_\emptyset}^{1/2})_{|M_\emptyset^{\sigma_s}}=\tilde{\mu}_{s|M_\emptyset^{\sigma_s}}$ where $\tilde{\mu}_s(x)=\tilde{\mu}(w_s^{-1}xw_s)$ for $x$ in $w_sHw_s^{-1}$.

\paragraph{}Thus, we get $s=\begin{pmatrix}&&1\\&\iddots&\\1&&\end{pmatrix}$ and $\mu\circ N_{rd,E}\left(\begin{pmatrix}a_1&&&&&\\&\ddots&&&&\\&&a_m&&&\\&&&a_m&&\\&&&&\ddots&\\&&&&&a_1\end{pmatrix}\right)=1$ for all $(a_1,\dots,a_m)$ 
in $(D^*)^m$. As $\begin{pmatrix}a_1&&&&&\\&\ddots&&&&\\&&a_m&&&\\&&&a_m&&\\&&&&\ddots&\\&&&&&a_1\end{pmatrix}$ is the embedding of $\begin{pmatrix}a_1&&\\&\ddots&\\&&a_m\end{pmatrix}$ of $\mathcal{M}_m(D\otimes_FE)$ in 
$\mathcal{M}_{2m}(D)$, it implies that $\mu\circ N_{rd,E}(D^*)=1$.\\
Let us show that $N_{rd,E}(D^*)=1$. There exists $L/F$ an extension of dimension $d$ such that $D\otimes_FL\simeq \mathcal{M}_d(L)$. As $d$ is odd, $L\otimes_FE$ is a field and it is a $d$-dimensional extension of $E$. 
Thus, we have the following natural commutative diagram :
$$\begin{matrix}D\otimes_FE&\longrightarrow&\mathcal{M}_d(L\otimes_FE)\\\uparrow&\circlearrowleft&\uparrow\\D&\longrightarrow&\mathcal{M}_d(L)\end{matrix}$$
which implies $N_{rd,E}(D^*)=N_{rd,F}(D^*)=F^*$. We deduce that $\mu_{|F^*}=1$.
\paragraph{}We end this proof by noticing that as $\sigma(P_\emptyset)=P_\emptyset^-$, then $P_\emptyset H$ is open in $G$ thanks to Proposition \ref{ouvert}.
\end{proof}

Now, we show that $St(1)$ is $\tilde{\mu}$-distinguished under $H$ if and only if $\mu_{|F^*}=1$ and $\mu\neq 1$ in the following two propositions :
\begin{prop}
If $\mu_{|F^*}=1$ and $\mu\neq 1$, then $St(1)$ is $\tilde{\mu}$-distinguished under $H$.
\end{prop}
\begin{proof}
Suppose $\mu_{|F^*}=1$ and $\mu\neq1$. We have $\sigma(P_\emptyset)=P_\emptyset^-$ and if $a$ is in 
$M_\emptyset^\sigma$, then $a$ is of the form 
$a=\begin{pmatrix}a_1&&&&&\\&\ddots&&&&\\&&a_m&&&\\&&&a_m&&\\&&&&\ddots&\\&&&&&a_1\end{pmatrix}$ with $a_1,\dots,a_m \in D^*$ and $\delta_{P_\emptyset}^{-1/2}\tilde{\mu}^{-1}(a)=1$. Thus, thanks to Theorem 2.8 of \cite{B-D.08}, 
$ind_{P_\emptyset}^G(1)$ is $\tilde{\mu}$-distinguished. We end the proof as in Proposition 3.6 of \cite{M.16}.
\end{proof}

Finally, we get the non-distinguished case : 

\begin{prop}
If $\mu=1$, then $St(1)$ is not $\tilde{\mu}$-distinguished under $H$.
\end{prop}
\begin{proof}
We do not give the proof because it is similar to the one of Theorem 3.1 in \cite{M.16}. 
\end{proof}

\section{Prasad and Takloo-Bighash conjecture}\label{epsilon}

\paragraph{}Let us summarize our results :

\begin{thm}
Let $n$ be a positive integer and let $\mu$ be a character of $E^*$. $E$ is embedded in $\mathcal{M}_n(D)$ if and only if $nd$ is even. We set $G=GL(n,D)$ and $St(1)=St(n,1)$ the Steinberg representation of $G$. 
We recall that $\tilde{\mu}$ denotes $\mu\circ N_{rd,E}$.
\begin{itemize}
\item If $d$ is even, $H=(C_{\mathcal{M}_n(D)}(E))^\times=GL(n,C_D(E))$ and $St(n,1)$ is $\tilde{\mu}$-distinguished under $H$ if and only if
\begin{itemize}
\item $\mu_{|F^*}=1$ and $\mu\neq1$ if $n$ is even.
\item $\mu=1$ if $n$ is odd.
\end{itemize}
\item If $d$ is odd and $n$ is even, $H=(C_{\mathcal{M}_n(D)}(E))^\times=GL(n/2,D\otimes_FE)$ and $St(n,1)$ is $\tilde{\mu}$-distinguished under $H$ if and only if $\mu_{|F^*}=1$ and $\mu\neq1$.
\end{itemize}
\end{thm}

\paragraph{}Now, let us write and prove the Conjecture 1 of \cite{P-TB} for the Steinberg representation $St(n,1)$ (which is a discrete series representation) :

\begin{thm}\label{TB}(Prasad and Takloo-Bighash conjecture, Steinberg case) Let $A=\mathcal{M}_n(D)$ and $\pi=St(n,1)$ which is an irreducible admissible representation of $A^\times=GL(n,D)=G$. Recall that $\pi$ corresponds via Jacquet-Langlands correspondence to 
$St(nd,1)$ (a representation of $GL(nd,F)$) with central character $\omega_\pi=1$. Let $\mu$ be a character of $E^\times$ such that $\mu^{\frac{nd}{2}}|_{F^\times}=\omega_\pi=1$. Then, the character $\mu\circ N_{rd,E}$ of 
$H=(C_{\mathcal{M}_n(D)}(E))^\times$ appears as a quotient in $\pi$ restricted to $H$ if and only if :
\begin{enumerate}
 \item the Langlands parameter of $\pi$ takes values in $GSp_{nd}(\mathbb{C})$ with similitude factor $\mu_{|F^\times}$.
 \item the epsilon factor satisfies $\epsilon(\frac{1}{2},\pi\otimes Ind_E^F(\mu^{-1}))=(-1)^n\omega_{E/F}(-1)^{\frac{nd}{2}}$ (where $\omega_{E/F}$ is the quadratic character of $F^\times$ with kernel the norms of $E^\times$).
\end{enumerate}

\end{thm}

As $N_{rd,F|H}=N_{E/F}\circ N_{rd,E}$, then if $\mu$ is a character of $E^*$, $\mu\circ N_{rd,E}$ can be extended to a character of $G$ if and only if there exists $\chi$ a character of $F^*$ such that $\mu=\chi\circ N_{E/F}$.

\paragraph{}Let us rephrase Theorem \ref{TB} in this case. Let $\pi$ and $\mu$ be as in the conjecture and suppose that there exists $\chi$ a character of $F^*$ such that $\mu=\chi\circ N_{E/F}$. We denote by $W_F$ the Weil group 
of $F$ and $BC_E$ denotes the base change to $E$.

\begin{itemize}
 \item As $\mu=\chi\circ N_{E/F}$, the statement ``$\mu\circ N_{rd,E}$ appears as a quotient in $\pi$ restricted to $H$'' is equivalent to saying that $St(n,1)$ 
 is $\chi\circ N_{rd,F}$-distinguished under $H$. This is again equivalent to $(\chi\circ N_{rd,F})^{-1}\otimes St(n,1)=St(n,\chi^{-1})$ is $H$-distinguished.
 \item Now, let us consider the 1\ts{st} point of the conjecture. The Langlands parameter of $\pi=St(n,1)$ is $Sp(nd)=:\Phi$ (where $Sp(nd)$ denotes the unique irreducible representation of $SL(2,\mathbb{C})$ of dimension $nd$) . The 1\ts{st} assertion in the conjecture means :
 \begin{eqnarray}
\text{there exists }<\cdot,\cdot> \text{ a nondegenerate alternating  bilinear form on } \mathbb{C}^{nd}=:V \text{ such that } \nonumber\\
\forall w\in W_F, \,\forall v,v'\in V, \hspace{0.5cm}<\Phi(w).v,\Phi(w).v'>=\mu_{|F^\times}(w)<v,v'>=\chi^2<v,v'>. \label{forme}
 \end{eqnarray}
As the Langlands parameter of $St(n,\chi^{-1})$ is $Sp(nd)\otimes \chi^{-1}=:\Psi$, statement \eqref{forme} is equivalent to :
\begin{eqnarray*}
 &&\text{there exists } <\cdot,\cdot> \text{ a nondegenerate alternating bilinear form on } V \text{ such that } \\
 &&\forall w\in W_F,\;\forall v, v'\in V,\hspace{.5cm} <\Psi(w).v,\Psi(w).v'>=<v,v'>
\end{eqnarray*}
which is exactly the definition of $St(n,\chi^{-1})$ being symplectic.
\item Finally, we consider the 2\ts{nd} point of the conjecture and we formulate the epsilon factor in another way :
\begin{eqnarray*}
 \epsilon(\frac{1}{2},\pi\otimes Ind_E^F(\mu^{-1}))&=&\epsilon(\frac{1}{2},Sp(nd)\otimes Ind_E^F((\chi\circ N_{E/F})^{-1}))\\
 &=&\epsilon(\frac{1}{2},Ind_E^F(Sp(nd)\otimes \chi_E^{-1}))\hspace{1cm}\text{where }\chi_E=\chi\circ N_{E/F}=\mu\\
 &=&\omega_{E/F}(-1)^{\frac{nd}{2}}\epsilon(\frac{1}{2},Sp(nd)\otimes \chi_E^{-1})\\
 &=&\omega_{E/F}(-1)^{\frac{nd}{2}}\epsilon(\frac{1}{2},BC_E(St(n,\chi^{-1})))
\end{eqnarray*}
so the 2\ts{nd} point of the conjecture is equivalent to : $\epsilon(\frac{1}{2},BC_E(St(n,\chi^{-1})))=(-1)^n$.
\end{itemize}

\paragraph{}To sum up, under the additional hypothesis that the character $\mu\circ N_{rd,E}$ of $H$ can be extended to a character of $G$, Theorem \ref{TB} is equivalent to the following, which is a reformulation similar to 
Conjecture 1.4 of \cite{FMW}:

\begin{thm}
 Let $St(n,\chi)$ be the Steinberg representation of $G=GL(n,D)$. $St(n,\chi)$ is $H$-distinguished if and only if it is symplectic and $\epsilon(\frac{1}{2},BC_E(St(n,\chi)))=(-1)^n$.
\end{thm}

\begin{rem}
 In \cite{FMW}, Conjecture 1.4 (a reformulation of Conjecture 1 of Prasad and Takloo-Bighash in \cite{P-TB}) is stated for general representations but only in the quaternionic case. It is checked for supercuspidal representations 
 with extra conditions.
\end{rem}

\begin{proof}[Proof of Theorem \ref{TB}]

We use the usual $\epsilon$, $\gamma$ and $L$-factors as defined in Godement-Jacquet (see \cite{G-J.72}); we omit the third parameter in $\epsilon$ and $\gamma$ which is a non-trivial additive character of $E$ but trivial on $F$. 

\paragraph{}According to the preceding reformulation (the three points above), we have to prove that $St(n,1)$ is $\tilde{\mu}$-distinguished if and only if $\mu_{|F^*}=1$ and $\epsilon(\frac{1}{2},Sp(nd)\otimes\mu^{-1})=(-1)^n$ 
that is to say $\epsilon(\frac{1}{2},St(nd,\mu^{-1}))=(-1)^n$ where $St(nd,\mu^{-1})$ is the Steinberg representation of $GL(nd,E)$ with parameter $\mu^{-1}$.

\paragraph{}For $s\in\mathbb{C}$,
$$\gamma(s,St(nd,\mu^{-1}))=(\mu^{-nd}(-1))^{nd-1}\epsilon(s,St(nd,\mu^{-1}))\frac{L(1-s,St(nd,\mu^{-1}))}{L(s,St(nd,\mu^{-1}))}.$$
        
\begin{sloppypar}On the other hand, $$\gamma(s,St(nd,\mu^{-1}))=\gamma(s+\frac{1-nd}{2},\mu^{-1})\times\gamma(s+\frac{1-nd}{2}+1,\mu^{-1})\times\dots\times\gamma(s+\frac{nd-1}{2},\mu^{-1}).$$\\

If $\mu_{|F^*}=1$, then $\bar{\mu}=\check{\mu}$ (where $\bar{\mu}$ denotes the Galois twist of $\mu$ and $\check{\mu}$ denotes the contragredient of $\mu$) so 
$$L(s,\mu^{-1})=L(s,\check{\mu}^{-1})\hspace{0.5cm}\text{and}\hspace{0.5cm}\epsilon(s,\mu^{-1})\epsilon(1-s,\mu^{-1})=1\hspace{0.5cm}\forall s\in \mathbb{C}.$$
Hence, if $\mu_{|F^*}=1$, $\gamma(s,\mu^{-1})=\epsilon(s,\mu^{-1})\frac{L(1-s,\mu^{-1})}{L(s,\mu^{-1})}\:\forall s\in\mathbb{C}$, so $\gamma(s,St(nd,\mu^{-1}))$ equals
\begin{equation*}
\begin{array}{l}
\epsilon(s+\frac{1-nd}{2},\mu^{-1})\times\dots\times\epsilon(s+\frac{nd-1}{2},\mu^{-1})\times\frac{L(1-s+\frac{nd-1}{2},\mu^{-1})\times\dots\times L(1-s+\frac{1-nd}{2},\mu^{-1})}{L(s+\frac{1-nd}{2},\mu^{-1})\times\dots\times L(s+\frac{nd-1}{2},\mu^{-1})}\\
=\frac{L(\frac{1}{2}+\frac{nd-1}{2},\mu^{-1})\times\dots\times L(\frac{1}{2}+\frac{1-nd}{2},\mu^{-1})}{L(\frac{1}{2}+\frac{1-nd}{2},\mu^{-1})\times\dots\times L(\frac{1}{2}+\frac{nd-1}{2},\mu^{-1})}.
\end{array}
\end{equation*}
Now for $\mu$ a character of $E^*$ such that $\mu_{|F^*}=1$ and $s$ a real number, we need to know when $L(s,\mu^{-1})$ has a pole. Let $\varpi_E$ be a uniformizer of $E$ and $q_E$ denote the cardinality of the residue field of $E$.\\
If $\mu$ is non-ramified, $L(s,\mu^{-1})=\frac{1}{1-\mu^{-1}(\varpi_E)q_E^{-s}}$ and $L(s,\check{\mu}^{-1})=L(s,\mu^{-1})=\frac{1}{1-\mu(\varpi_E)q_E^{-s}}$ so $L(s,\mu^{-1})$ has a pole is equivalent to :
$$\mu^{-1}(\varpi_E)=q_E^s=\mu(\varpi_E)\Leftrightarrow s=0\text{ and }\mu(\varpi_E)=1\textit{ i.e. }\mu=1$$
\begin{flushleft}because $\mu$ is non-ramified.\end{flushleft}
If $\mu$ is ramified, $L(s,\mu^{-1})=1$ so $L(s,\mu^{-1})$ has no pole.\\
To conclude, $L(s,\mu^{-1})$ has a pole if and only if $s=0$ and $\mu=1$.

\paragraph{}Finally, if $\mu_{|F^*}=1$, $\gamma(\frac{1}{2},St(nd,\mu^{-1}))=\epsilon(\frac{1}{2},St(nd,\mu^{-1}))$ and also 
$\gamma(\frac{1}{2},St(nd,\mu^{-1}))=\frac{L(\frac{nd}{2},\mu^{-1})\times\dots\times L(1-\frac{nd}{2},\mu^{-1})}{L(1-\frac{nd}{2},\mu^{-1})\times\dots\times L(\frac{nd}{2},\mu^{-1})}$ so : 
\begin{itemize}
\item ($\mu_{|F^*}=1$ and $\mu\neq1$) $\Leftrightarrow$ ($\epsilon(\frac{1}{2},St(nd,\mu^{-1}))=1$ and $\mu_{|F^*}=1$).
\item ($\mu=1$) $\Leftrightarrow$ ($\epsilon(\frac{1}{2},St(nd,\mu^{-1}))=\underset{s\rightarrow \frac{nd}{2}}{lim}\frac{L(\frac{nd}{2}-s,\mu^{-1})}{L(s-\frac{nd}{2},\mu^{-1})}=\underset{s\rightarrow 0}{lim}\frac{1-q_E^{-s}}{1-q_E^s}=\underset{s\rightarrow 0}{lim}\frac{-1}{q_E^s}=-1$ 
\\and $\mu_{|F^*}=1$).
\end{itemize}
This is equivalent to $St(n,1)$ is $\tilde{\mu}$-distinguished if and only if $\mu_{|F^*}=1$ and $\epsilon(\frac{1}{2},St(nd,\mu^{-1}))=(-1)^n$.
\end{sloppypar}
\end{proof}

\bibliographystyle{alpha}
\bibliography{bib_distinction}

\begin{thebibliography}{FMW17}

\bibitem[BD08]{B-D.08}
P.~Blanc and P.~Delorme.
\newblock Vecteurs distributions {H}-invariants de représentations induites,
  pour un espace symétrique réductif p-adique {G/H}.
\newblock {\em Ann. Inst. Fourier (Grenoble)}, 58:1, 2008.

\bibitem[BZ76]{BZ.76}
J.~N. Bernstein and A.~V. Zelevinsky.
\newblock Representations of the group {GL}(n,{F}), where {F} is a local
  non-{A}rchimedean field.
\newblock {\em Uspekhi Mat. Nauk}, 31:3, 1976.

\bibitem[BZ77]{BZ.77}
J.~N. Bernstein and A.~V. Zelevinsky.
\newblock Induced representations of reductive p-adic groups.
\newblock {\em Ann. Sc. E.N.S.}, 1977.

\bibitem[FMW17]{FMW}
Brooke Feigon, Kimball Martin, and David Whitehouse.
\newblock Periods and nonvanishing of central l-values for gl(2n).
\newblock {\em to appear in Israel J. Math.}, 2017.

\bibitem[GJ72]{G-J.72}
R.~Godement and H.~Jacquet.
\newblock Zeta functions of simple algebras.
\newblock {\em Lecture Notes in Mathematics. Springer-Verlag, Berlin-New York},
  260, 1972.

\bibitem[HW93]{H-W.93}
A.~Helminck and S.~Wang.
\newblock On rationality properties of involutions of reductive groups.
\newblock {\em Adv. Math.}, 99:1, 1993.

\bibitem[KT08]{K-T.08}
S.~Kato and K.~Takano.
\newblock Subrepresentation theorem for p-adic symmetric spaces.
\newblock {\em Int. Math. Res. Not. IMRN}, 2008:11, 2008.

\bibitem[Mat11]{M.11}
Nadir Matringe.
\newblock Distinguished generic representations of {GL}(n) over p-adic fields.
\newblock {\em Int. Math. Res. Not. IMRN}, 2011:1, 2011.

\bibitem[Mat15]{M.15}
Nadir Matringe.
\newblock On the local {B}ump-{F}riedberg {$L$}-function.
\newblock {\em J. Reine Angew. Math.}, 709:119--170, 2015.

\bibitem[Mat16]{M.16}
Nadir Matringe.
\newblock Distinction of the {S}teinberg representation for inner forms of
  {G}{L}(n).
\newblock {\em to appear in Math. Z.}, 2016.

\bibitem[PTB11]{P-TB}
Dipendra Prasad and Ramin Takloo-Bighash.
\newblock Bessel models for {GS}p(4).
\newblock {\em J. Reine Angew. Math.}, 655:189--243, 2011.

\end{thebibliography}
\end{document}